\documentclass[11pt,a4paper]{article}
\usepackage{lineno,hyperref}
\modulolinenumbers[5]
\usepackage[english]{babel}
\usepackage[a4paper,tmargin=3cm,bmargin=3cm,rmargin=2.5cm,lmargin=2.5cm]{geometry}
\usepackage{amsfonts,amsmath,amssymb}
\usepackage[T1]{fontenc}
\usepackage{enumerate}
\usepackage{graphicx}
\usepackage[utf8]{inputenc}
\usepackage{color}
\usepackage{amsthm}
\usepackage{t1enc}
\usepackage[english]{babel}
\numberwithin{equation}{section}
\newcommand{\disp}{\displaystyle}

\newtheorem{theorem}{Theorem}[section]

\newtheorem{proposition}[theorem]{Proposition}
\newtheorem{corollary}[theorem]{Corollary}
\newtheorem{definition}[theorem]{Definition}
\newtheorem{example}[theorem]{Example}

\newtheorem{Remarque}[theorem]{Remark}
\newtheorem{property}[theorem]{Property}
\begin{document}
\date{}
\title{\LARGE The First Eigenvalue of the Kohn-Laplace
 Operator in the Heisenberg Group\\}

\author{  Najoua Gamara\footnote{Corresponding authors; ngamara7@gmail.com} \;\; Akram  Makni}
 \maketitle
 \date{ }
\begin{abstract}
 In this paper, by extending  the notions of harmonic transplantation and harmonic radius in the Heisenberg group,  we
  give an upper bound for the first eigenvalue for the following Dirichlet problem:
$$(P_{\Omega})
\left\{ \begin{array}{lllll}
      -\Delta_{\mathbb{H}^1} u & = & \lambda u & \mbox{in} & \Omega  \\
     u & = & 0 & \mbox{on} & \partial \Omega,
      \end{array} \right.$$
      where $ \Omega $ is a regular bounded domain of $ \mathbb{H}^1$ with smooth boundary and $\Delta_{\mathbb{H}^1}$ is the Kohn-Laplace operator.
      Using  the results  of  P.Pansu given in $ \cite{6, 7},$  which give the relation between  the volume of $\Omega$  and the
perimeter of its boundary.
      we prove the following result

$$ \lambda_{1}( \Omega ) \leq C_{\Omega} \displaystyle \frac{ l_{11}^2 }{ \displaystyle \max_{ \xi \in \Omega } r_{\Omega}^2(\xi)} $$
 where $l_{11} $ is the first strictly positive zero of the Bessel function of first kind and
order 1, $ C_{\Omega} $ is a constant depending of $ \Omega $ and
$r_{\Omega}(\xi) $ is  the harmonic radius of $ \Omega $ at a point
$\xi$ of $\Omega.$
\end{abstract}
\vspace*{0.5cm}
  \vspace*{0.5cm}

\textbf{MSC $2010$:} 35J08, 35P15, 35R03, 47J10, 49J35, 49J40.

 \tableofcontents
 \section{ Introduction }

The method of harmonic transplantation as proposed by Hersch in $
\cite{1} $ is a technique to construct
 test functions for variational problems of the form

$$ J\left[ \Omega \right]  = \displaystyle  \inf_{ u \in H_{0}^{1} }    \frac{\int_{\Omega}  \left| \nabla u \right|^2 \ dy}{\int_{\Omega}  \left|  u \right|^2 \ dy},  $$

where $ \Omega $ is a domain of $ \mathbb{R}^{n} $, $ n \geq 3 $
with smooth boundary. The harmonic transplantation is a
generalization of the conformal transplantation for simply connected
planar domain.
 This method has numerous applications in mathematical physics and function theory
  $\cite{1,2,3} $.
 In particular it has been used to give estimates from above of the   functional $ J \left[ \Omega \right],$  in contrast to the method
  based on symmetrization techniques which gives estimates from below.
The harmonic transplantation is connected to the Green's function $
G_{\left\{\Omega; x \right\}} $ of $ \Omega $ with Dirichlet
 boundary condition, which is the solution of

$$ \left\{ \begin{array}{lllll}
 \Delta G_{\left\{\Omega; x \right\}} & = & -\delta_{x} & \mbox{in} & \Omega  \\
 G_{\left\{\Omega; x \right\}} & = & 0 & \mbox{on} & \partial \Omega.
 \end{array} \right. $$

The Green's function can be decomposed as :

 $$  G_{\left\{\Omega; x \right\}}(y) = \gamma \left( F_{x}(y) - H_{\left\{\Omega; x \right\}}(y)\right)  $$

 with

 $$ \gamma = \left\lbrace \begin{array}{ll}

 \displaystyle \frac{1}{2\pi} & n = 2 \\ \\

 \displaystyle \frac{\Gamma(n/2)}{2(n-2)\pi^{n/2}} & n \geq 3

 \end{array} \right. $$

$$  F_{x}(y) = F(\|y- x\|)= \left\lbrace \begin{array}{ll}

- \ln ( \| y-x \|) &  n = 2 \\ \\

        \| y-x \|^{2-n} & n \geq 3

        \end{array} \right.   $$

and $ H_{\left\{\Omega; x \right\}}$ is the solution of

$$ \left\{ \begin{array}{lllll}
 \Delta H_{\left\{\Omega; x \right\}} & = & 0 & \mbox{in} & \Omega  \\
 H_{\left\{\Omega; x \right\}} & = & F_{x} & \mbox{on} & \partial \Omega ,
 \end{array} \right. $$
 The functions $  F_{x}$ and $ H_{\left\{\Omega; x \right\}} $ are  called respectively the singular part and the regular part of
 $ G_{\left\{\Omega; x \right\}} $.
The leading term of the regular part of the Green's function $$
t(x):=  H_{\left\{\Omega; x \right\}}(x)$$ is called  Robin's
function
 of $\Omega$ at $x.$
 For $ n\geq 3,$ the harmonic radius $ r(x) $ of $ \Omega $ at $ x $ defined in  $ \cite{4} $   by
 $$F(r(x))= t(x)$$
 which gives
 $$  r(x) =  \displaystyle   H_{\left\{\Omega; x \right\}}^{ \frac{1}{2-n}}(x) .$$
 The harmonic radius vanishes on the boundary  $\partial\Omega$ and takes its maximum at the so called harmonic centers of $\Omega.$

 Let $\Omega$ be a domain  of $ \mathbb{R}^{n} $ with smooth boundary, $ n \geq 3 $ and fix $x\in \Omega.$
 If $ u : B(0, r(x)) \rightarrow \mathbb{R}_{+} $ is a radial function and $ \mu $ is defined by
 $ u(y) = \mu \left( G_{\left\{B(0, r(x)); 0 \right\}}(y) \right) ,$ then

     $$ U(y) = \mu \left( G_{\left\{\Omega; x \right\}}(y) \right) $$

defines the harmonic transplantation of $ u $ into $ \Omega .$ \\
The harmonic transplantation has the following properties:
\begin{enumerate}
\item It preserves the Dirichlet integral
\begin{equation} \label{hersh1} \int_{B(0, r(x))}  \left| \nabla u  \right|^2 \ dy  = \int_{\Omega}  \left| \nabla U \right|^2 \ dy \end{equation}

\item For  every positive function $ \varphi $

\begin{equation} \label{hersh2} \int_{B(0, r(x))}  \varphi(u)  \ dy  \leq \int_{\Omega}  \varphi(U) \ dy .\end{equation}
\end{enumerate}
The first eigenfunction of the Dirichlet problem on $ B(0, r(x)) $
is radial and positive, then as proved by Hersch in $ \cite{4},$ by
using Rayleigh's
 principle, ($ \ref{hersh1} $) and ($ \ref{hersh2} $), one obtain  $ \forall x \in \Omega $

$$ \lambda_{1}(\Omega) \leq \frac{\int_{\Omega}  \left| \nabla U \right|^2 \ dy}{\int_{\Omega}  \left|  U \right|^2 \ dy} \leq \frac{ \int_{B(0, r(x))}  \left| \nabla u  \right|^2 \ dy }{\int_{B(0, r(x))}  \left| u  \right|^2 \ dy } = \lambda_{1}(B(0, r(x))), $$

hence

\begin{equation} \label{lambda1}  \lambda_{1}(\Omega) \leq \lambda_{1}(B(0, \max_{\Omega} r(y))). \end{equation}

Recall that the harmonic transplantation was also extended  to
spaces of constant curvature by C. Bandle, A. Brillard, and M.
Fluscher, one can see $ \cite{5}.$\\

In this work, we will extend the notions of harmonic transplantation
and harmonic radius to the Heisenberg group $ \mathbb{H}^1.$ More
precisely, we will define counter parts of the harmonic
transplantation and the harmonic radius in the subriemannian
settings. Our aim is to give an upper bound for the first eigenvalue
for the following Dirichlet problem:

$$(P_{\Omega})
\left\{ \begin{array}{lllll}
      -\Delta_{\mathbb{H}^1} u & = & \lambda u & \mbox{in} & \Omega  \\
     u & = & 0 & \mbox{on} & \partial \Omega,
      \end{array} \right.$$

where $\Omega$ is a regular bounded domain of   the Heisenberg group
$ \mathbb{H}^1$ with smooth boundary.\\
 The existence of eigenvalues
for the Dirichlet problem $(P_{\Omega})$ has been proved by X.Luo
  and P.Niu in $ \cite{10,11,12} $ using the Kohn inequality (one can see $ \cite{13}$) for the vector fields generating the horizontal
   tangent bundle of the Heisenberg group  together with the spectral properties of compact
   operators.\\
The method, we will use in the present work
 is based on technics related to the harmonic transplantation method
 and a result of P.Pansu  given in $ \cite{6, 7} $ which give the relation between  the volume of $\Omega$  and the
perimeter of its boundary. But, we know that we won't lead to a
similar formula to ($ \ref{lambda1}$)
  since Heisenberg balls or Koranyi balls are not isoperimetric regions in the Heisenberg
group $ \mathbb{H}^1.$\\

  The isoperimetric problem as introduced by P.Pansu in $ \cite{6, 7} $ consists to find  subsets $ I $ of  $ \mathbb{H}^1 $ among all subsets $ I' $  with the same
   Riemannian volume as $ I $ associated to the  left invariant Riemannian metric such that

 \begin{equation} \label{pansu1} P_{\mathbb{H}^1}(I) \leq P_{\mathbb{H}^1}(I') \end{equation}

  where $ P_{\mathbb{H}^1} ( I ) $ denotes  the horizontal perimeter of $ I.$  The set $ I $ is called an isoperimetric
   region in $ \mathbb{H}^1.$
   The existence of isoperimetric regions was proved by G. P. Leonardi and S. Rigot in $ \cite{9} $ for the more general context
    of Carnot groups.\\
For any bounded open set $ \Omega $ in $ \mathbb{H}^1 $ with $ C^1 $
boundary, P.Pansu proved in $ \cite{6, 7} $ the existence of a
constant $ C
> 0 ,$ so that
\begin{equation} \label{pansu2}  \vert \Omega \vert^{ \frac{3}{4}}  \leq C P_{\mathbb{H}^1}( \Omega ). \end{equation}

and  conjectured that any isoperimetric region of $ \mathbb{H}^1 $
 bounded by a smooth surface is congruent to a bubble set:

$$ \mathbb{B} (0, R) = \left\{ (z, t) \in \mathbb{H}^1 , |t| < R^2 arccos \left( \frac{ |z|}{R}\right) + |z| \sqrt{R^2 - |z|^2}, \ |z| < R \right\}. $$
for some $R>0$. Recall that $ \mathbb{B} (0, R) $ is obtained by
rotating around the $ t $-axis the geodesic joining  the two poles
 $ -\frac{\pi}{2} R^2 $ and $ \frac{\pi}{2} R^2 $.  The boundary of $ \mathbb{B} (0, R) $  is

 $$\mathbb{S}_{R}= \left\{ (z, t) \in \mathbb{H}^1 , |t| = R^2 arccos \left( \frac{ |z|}{R}\right) + |z| \sqrt{R^2 - |z|^2}, \ |z| < R \right\} .$$

In $ \cite{8} ,$  M.Ritore and C.Rosales proved that an
isoperimetric region in the class of Heisenberg group sub domains
with  $ C^2 $ surfaces
  as boundaries is congruent to a subset of $\mathbb{H}^1$  having $\mathbb{S}_{R}$ as boundary. A similar result was
  obtained by R. Monti and M. Rickly  in $\cite{24}$ for convex euclidian domains $\mathbb{H}^1.$

 In the present work: we won't use isoperimetric regions and the isoperimetric
 inequality in the Heisenberg group $\mathbb{H}^1$ with the method of harmonic transplantation to determine an upper
bound of the first eigenvalue $\lambda_{1}(\Omega)$ of
$(P_{\Omega})$ as it was the case for the Euclidian settings.
Instead, we will use the result of P.Pansu $\cite{6, 7}$ given above
by \eqref{pansu2}. Our main results are\\\\
\textbf{Theorem }\ref{eigenvalue}\\
 Let $ \Omega $ be a  bounded
domain of $ \mathbb{H}^1 $ with smooth boundary, such that $\Omega$
and $\mathbb{H}^{1}\setminus \bar{\Omega}$ satisfy the uniform
exterior ball  property (\ref{boulexter}) and let $\xi$ be any point
in $\Omega,$ we have

   $$ \lambda_{1}( \Omega ) \leq    C_{\Omega}(\xi) \displaystyle \frac{l_{11}^2}{r_{\Omega}^2(\xi)}$$

   where $ l_{11} $ is the first strictly positive zero of the Bessel function $ J_{1}(l)$  and $ r_{\Omega}(\xi) $ is the harmonic radius
   of $ \Omega $ at $ \xi.$\begin{flushright}
    $\square$
\end{flushright}

    \textbf{Theorem }\ref{Harmonic center}\\
Let $ \Omega $ be a  bounded domain of $ \mathbb{H}^1 $ with smooth
boundary, such that $\Omega$ and $\mathbb{H}^{1}\setminus
\bar{\Omega}$ satisfy the uniform exterior ball  property
(\ref{boulexter}) and let $\xi$ be any point in $\Omega,$ we have
   $$ \lambda_{1}( \Omega ) \leq C_{\Omega}\frac{ l_{11}^2 }{ \displaystyle \max_{ \xi \in \Omega } r_{\Omega}^2(\xi)}$$

   with $ C_{\Omega} = \displaystyle \min_{ \xi \in A_{\Omega}} C_{\Omega}(\xi),$ where  $A_{\Omega}$ is the set of harmonic centers of
   $\Omega.$\begin{flushright}
    $\square$
\end{flushright}
The second theorem is a refined version of the first one.\\

\vspace*{0.2cm}

\par The paper is organized as follows: in Section 2, we will introduce the general framework, we begin by recalling the local structure
 of the Heisenberg group, the Carnot-Carath\'eodory distance and the Koranyi balls. The section 3 is devoted to the study of  some properties
 of  Green's functions of the Kohn- Laplace operator and their level sets. These properties lead to extend the notion of harmonic radius to
  the Heisenberg group $ \mathbb{H}^1,$ which  is the purpose of section 4.  In section 5,  we will use the  Green's functions of  Koranyi
   balls to extend the harmonic transplantation to the Heisenberg group $ \mathbb{H}^1 .$ Finally, in section 6,  we  prove our main result,
    Theorem \ref{eigenvalue}, then we give a refined version of this result which is   Theorem \ref{Harmonic center}. For the proof,
    we proceed as follows,
     we  perturb the Dirichlet problem $(P_{\Omega})$ on a Koranyi's ball, then we compare the Rayleigh
    ratios of its solutions with those of their corresponding harmonics transplantations.

\begin{flushright}
    $\square$
\end{flushright}

\textbf{ACKNOWLEDGEMENTS.}\\
 The authors would like to
thank P. Pansu for his suggestions and advices.

\section{The Heisenberg group $ \mathbb{H}^1 $ }

The Heisenberg group $ \mathbb{H}^1 $  is the Lie group whose
underlying manifold is $\mathbb{C}\times \mathbb{R}= \mathbb{R}^3$,
with coordinates
 $(x, y, t),$ the group law  $ o $  is defined as follows:
for all $ \eta = (x, y, t) $, $ \xi = (x_{0}, y_{0}, t_{0}) $ $ \in
\mathbb{R}^3 ,$  $\eta o \xi = (x+x_{0}, y+y_{0},
t+t_{0}+2(x_{0}y-xy_{0}) ) .$ For $ \eta \in \mathbb{H}^1 $, the
left translation  by $ \xi $ is the diffeomorphism $ L_{\eta}({\xi})
= \eta o \xi .$ The Heisenberg group dilations are the $
\mathbb{H}^1- $ transformations:

$$  \delta_{\lambda}(\eta) = ( \lambda x, \lambda y, \lambda^2 t),  \lambda > 0 .$$

The homogeneous norm of the space is given for all $\eta \in
\mathbb{H}^1 $ by
$$  \rho(\eta) = \left( ( x^2 + y^2 )^2 + t^2 \right)^{ \frac{1}{4} }, $$

and the natural distance is accordingly defined for all $\eta , \xi
\in \mathbb{H}^1 $ by:
$$ d_{\mathbb{H}^1}( \xi, \eta ) =  \rho({\xi}^{-1} o \eta). $$

The open balls $ B(\xi, r)$ associated to this distance are called
the Koranyi  balls
$$ B(\xi, r) = \left\{ \eta \in \mathbb{H}^1, d_{\mathbb{H}^1}( \xi, \eta ) < r \right\} .$$
 A function $ u : \Omega \rightarrow \mathbb{R} $ defined on a
domain $\Omega $  of $\mathbb{H}^1$ is called radial or $ \rho
$-radial if there exists a function $ \varphi : \mathbb{R}^{+}
\rightarrow \mathbb{R} $
  such that $ u(\xi) = \varphi(\rho(\xi)) ,$ for all $ \xi \in \Omega $.

The following left-invariant vector fields
$$ X = \frac{\partial}{\partial x} + 2y \frac{\partial}{\partial t}, \ Y = \frac{\partial}{\partial y} - 2x \frac{\partial}{\partial t}, \ T =  \frac{\partial}{\partial t} $$
give a base for the tangent space $ T_{\eta} \mathbb{H}^1 $. Whereas
The horizontal distribution $ H $ of  $ \mathbb{H}^1 $ is
 generated  by the vector fields $ X $ and $ Y .$  The
horizontal projection  of a vector $ V $ on $ H $
 will be denoted by $ V_{H} $, so a vector $ V $ is called horizontal if $ V = V_{H} $.
  For any $ C^1 $ function $ \varphi $ defined in an open set of $  \mathbb{H}^1 ,$
  the  horizontal gradient of $\varphi$ is given by

\begin{equation} \label{gradient} \nabla_{\mathbb{H}^1} \varphi = (X \varphi) X + (Y \varphi) Y. \end{equation}
We denote the  Lie bracket of two vectors fields $ V_{1} $ and $
V_{2} $ defined on  $ \mathbb{H}^1 $  by $ \left[ V_{1}, V_{2}
\right] $.\\ We have, the following  identities for the base $(X, Y,
T)$ $$ \left[ X, T\right] = \left[ Y, T \right] = 0 \,\,
\text{and}\,\,\,
   \left[X, Y\right] = -4T .$$
We consider on $  \mathbb{H}^1 ,$ the left-invariant Riemannian
metric $ g_{\mathbb{H}^1} = <.,.> $ given in $\cite{27}$

$$  g_{\mathbb{H}^1} = dx^2 + dy^2 + ( dt - 2ydx + 2xdy )^2  $$

for which the triplet  $ ( X, Y, T ) $ is an orthonormal basis for
all point in $  \mathbb{H}^1$. The corresponding Levi Civita
connexion  denoted by $ D $ satisfies

$$  \begin{array}{lll}

 D_{X}X = 0, & D_{Y}Y = 0, & D_{T}T = 0 \\

 D_{X}Y = -2T, & D_{X}T = 2Y, & D_{Y}T = -2X \\

 D_{Y}X = 2T, & D_{T}X = 2Y, & D_{T}Y = -2X.

 \end{array} $$

\begin{flushright}
    $\square$
\end{flushright}

 The volume of a domain $ \Omega \subset \mathbb{H}^1 $ denoted by $ \left| \Omega \right| $, is the Riemannian volume associated to the
 the left invariant Riemannian metric $ g_{\mathbb{H}^1} $, which coincides with the Lebesgue measure in $  \mathbb{R}^3 $.\\
   For  a $ C^1 $ surface $ \Sigma $ in $ \mathbb{H}^1 $ with  $ N $ as unit vector normal to $ \Sigma ,$ we define its area by

 \begin{equation} \label{area} A(\Sigma) = \int_{\Sigma} \left| N_{ \mathbb{H}^1 } \right| d\Sigma, \end{equation}

 where $ N_{H} = N - < N, T> T $ denotes the orthogonal projection of $ N $ onto the horizontal distribution, and $ d\Sigma $ is
 the Riemannian area element induced by the metric $ g_{\mathbb{H}^1} $ on $ \Sigma $. If $ \Sigma $ is a $ C^1 $ surface of a bounded domain
  $ \Omega $, so $  A(\Sigma) $ coincides with $ P_{\mathbb{H}^1} ( \Omega ) $ the horizontal perimeter of $ \Omega $
  is defined by

 \begin{equation} \label{ph1} P_{\mathbb{H}^1} ( \Omega ) = \sup \left\{ \int_{ \Omega } div V dv ; \ V \mbox{ horizontal of class} C^1, \ \left| V \right| \leq 1 \right\}, \end{equation}

 where $ div V $ is the Riemannian divergence of the vector field $ V $, and $ dv $ is the element of volume associated to $ g_{\mathbb{H}^1} $.
  For more details one can see $ \cite{14,15,23} .$\begin{flushright}
    $\square$
\end{flushright}

 \begin{Remarque}

 In the rest of the paper, for  a given domain $ D \subset \mathbb{H}^1 $ with boundary
 $ \partial D $, $ d\Sigma_{D} $ will  denote the Riemannian element area on $ \partial D $
 induced by the Riemannian metric $ g_{\mathbb{H}^1}$.

 \end{Remarque}

 We denotes by $ \Delta_{\mathbb{H}^1} $ the Kohn-Laplace  operator on the Heisenberg group,
 which is  defined by

 $$ \Delta_{\mathbb{H}^1} = X^2 + Y^2. $$

 The Kohn-Laplace operator $ \Delta_{\mathbb{H}^1} $ satisfies the  H\"{o}rmander conditions,
 thus it is  hypoelliptic on $ \mathbb{H}^1 $.  Particularly, any harmonic function $ u $ defined on an
 open set $ \Omega $ of $\mathbb{H}^1 $  ($ \Delta_{\mathbb{H}^1} u = 0,\, \text{in} \, \Omega )$ is $ C^{\infty}.$

A fundamental solution of $-\Delta_{\mathbb{H}^1}$  with pole at
zero is given by $\displaystyle{\Gamma(\xi)= \frac{1}{8\pi
\rho(\xi)^{2}}}$. Moreover, a fundamental solution with pole at
$\xi$, is $\displaystyle{
\Gamma(\xi,\xi^{'})= \frac{1}{8\pi d(\xi,\xi^{'})^{2}}}$. \\
A basic role in the functional analysis on the Heisenberg group is
played by the following Sobolev-type inequality:
\begin{eqnarray*}
| \varphi |^{2}_{4} \leq c | \nabla_{\mathbb{H}^{1}} \varphi
|^{2}_{2},\,\, \forall \varphi \in
\mathcal{C}_{0}^{\infty}(\mathbb{H}^{1})
\end{eqnarray*}
 This inequality ensures
in particular that for every domain $\Omega$ of $\mathbb{H}^{1}$, $
\displaystyle{ | \varphi | = |\nabla_{\mathbb{H}^{1}} \varphi |_{2}}
$ is a norm on $\mathcal{C}_{0}^{\infty}(\Omega)$. We shall denote
by $S^{1,2}(\Omega)$ the Banach space of the functions $u \in
L^{2}(\Omega)$
  such that the distributional derivative $X_{1}u,\;X_{2}u \in L^{2}(\Omega),$
 and by  $S^{1,2}_{0}(\Omega)$  the closure of
$\mathcal{C}_{0}^{\infty}(\Omega)$ with respect to the norm above.
 The space $S^{1,2}_{0}(\Omega)$ is a Hilbert space with the inner
product $ <u, v
>_{S^{1,2}_{0}(\Omega)}=\int_{\Omega}<\nabla_{\mathbb{H}^{1}}
u,\nabla_{\mathbb{H}^{1}} v > .$ \\
In what follows, we let
$ 0 < \lambda_{1} \leq \lambda_{2} \leq ... \leq \lambda_{m} \leq ...\longrightarrow + \infty $ denote the successive eigenvalues for $(P_{\Omega})$ with corresponding eigenfunctions:  $u_{1}, u_{2}, . . . , u_{m}, ...  $ in $ S^{1,2}(\Omega)$.\\
 The Kohn-Laplace  operator is related to the following result called the maximum principle
 \begin{theorem}{$\cite{30}$} \label{maximumprinciple}
 Let $ \Omega $ be an open bounded  set of $ \mathbb{H}^1 $ and $ u \in C^2 \cap C(\overline{\Omega}) $. If $ \Delta_{\mathbb{H}^1} u\geq 0, $
 then
 $$ \sup_{\overline{\Omega}} u = \sup_{\partial \Omega} u, $$
 and if $ \Delta_{\mathbb{H}^1} u \leq 0,$
 $$ \inf_{\overline{\Omega}} u = \inf_{\partial \Omega} u. $$
 \end{theorem}\begin{flushright}
    $\square$
\end{flushright}
 We will end this section by giving the volumes of  Koranyi balls.
  \begin{proposition}
 For $ R > 0 $, the volume of the Koranyi's ball $B(0, R)$ is given
 by
 \begin{equation} \label{ballvolume} \left| B(0, R) \right| = \frac{\pi^2}{2} R^4. \end{equation}\begin{flushright}
    $\square$
\end{flushright}

 \end{proposition}

 \begin{proof}

We consider the function

 $$ \begin{array}{llll}

 \phi & ]0, 1[ \times ]0, \pi[ \times ]0, 2\pi[ & \longrightarrow & B(0, R) \backslash \textit{P} \\

      & (r, \theta, \varphi) & \longmapsto & ( r \sqrt{sin\theta} cos\varphi, r \sqrt{sin\theta} sin\varphi, r^2 cos \theta),
      \end{array} $$

 where

  $$ \textit{P} = \left\{ (x, y, t) \in \mathbb{H}^1, y = 0, x \geq 0 \right\}. $$

  So $ \phi $ is a $ C^1 $-diffeomorphism and $ \forall (r, \theta, \varphi) \in ]0, 1[ \times ]0, \pi[ \times ]0, 2\pi[ ,$ we have

  $$ det\left(J_{\phi}(r, \theta, \varphi)\right) = r^3. $$

  We observe that the Lebesgue measure of the set $ \textit{P} $ is zero, so we conclude that

  $$ \left| B(0, R) \right| = \int_{0}^{1} \int_{0}^{\pi} \int_{0}^{2\pi}
   \left| det\left(J_{\phi}(r, \theta, \varphi)\right) \right| dr d\theta d\varphi =
    \frac{\pi^2}{2} R^4. $$

  \end{proof}

\section{  Green's function}

Before giving the different properties of the Green's function, we
  will introduce the notion of uniform exterior ball property.

  \begin{definition}

We say that the domain $\Omega$ of $\mathbb{H}^{1}$ satisfies the
uniform exterior ball property, if the following condition holds
\begin{eqnarray}\label{boulexter} \left\{\begin{array}{l}
There\;exists\;r_{0}>0 \; such \;that\;
\\ \forall \;\xi\; \in\;\partial\;\Omega ,\;\;\forall\; r\;\in\;]0, r_{0}]
\;\exists\; \eta\;\in\;\mathbb{H}^{n}\; such\; that\;\\
B_{d}(\eta, r)\cap\Omega = \emptyset\;\; and \;\;\xi\in\partial
B_{d}(\eta, r).
\end{array}
\right.
\end{eqnarray}

\end{definition}
\begin{flushright}
    $\square$
\end{flushright}

We consider an open set $\Omega$ of the Heisenberg group $
\mathbb{H}^{1}$ with smooth boundary, such that $\Omega$ and
$\mathbb{H}^{n}\setminus\overline{\Omega}$ satisfy
 the uniform exterior ball property. This condition seems to be natural since the
Koranyi balls of $\mathbb{H}^{1}$ satisfy such a property. In
\cite{31}, Hansen and Hueber proved that this condition implies that
the domain $\Omega$  is then regular: for $\varphi$ in
$C(\partial\Omega)$, the Dirichlet problem
$$ \left\{ \begin{array}{lllll}
 \Delta_{\mathbb{H}^1} u & = & 0 & \mbox{in} & \Omega  \\
 u & = & \varphi & \mbox{on} & \partial \Omega.
 \end{array} \right. $$
has a classical solution $u \in C^2(\Omega)\cap
C(\overline{\Omega})$. In particular, $\Omega $ has a Green's
function for every $\xi \in \Omega$, which we denote by
$G_{\left\{\Omega; \xi \right\}}.$ More precisely, for every $\xi
\in \Omega,$ there exists a classical solution $H_{\left\{\Omega;
\xi \right\}}$ of the Dirichlet problem
\begin{eqnarray}\label{DericG}
\left\{
\begin{array}{lllll}
\Delta_{\mathbb{H}^{1}} H_{\left\{\Omega; \xi \right\}} &=&0 &  in\;\;  \Omega\\
H_{\left\{\Omega; \xi \right\}}&=& \Gamma_{\xi} &  on\;\;
\partial\Omega,
\end{array}\right.
\end{eqnarray}
where $\Gamma_\xi$ is a fundamental solution of
$-\Delta_{\mathbb{H}^{1}}$ with pole at $\xi.$ Hence, the Green's
function is defined by
 \begin{equation} \label{greenfunction} G_{\left\{\Omega; \xi \right\}}(\eta) = \Gamma_{\xi}(\eta) - H_{\left\{\Omega; \xi \right\}}(\eta) \end{equation}
 with
\begin{equation} \label{gama}  \Gamma_{\xi}(\eta) = \frac{1}{8\pi \rho^2 ( {\xi}^{-1} o \eta ) }. \end{equation}
  The functions $\Gamma_{\xi}$ and   $ H_{\left\{\Omega; \xi
\right\}} $ are called respectively  the singular part and the
regular part of the  Green's function $ G_{\left\{\Omega; \xi
\right\}}.$\begin{flushright}
    $\square$
\end{flushright}

 \subsection{Examples}
Finding Green's functions for given domains is difficult, however
for certain domains with nice geometries, it is possible.  We will
give two examples below.
 \begin{example} The Green's function of the Koranyi ball $ B(\xi, R) $ with pole at $\xi$
 is given by

 $$ G_{\left\{B(\xi, R); \xi \right\}}(\eta) = \frac{1}{8\pi \rho^2 ( {\xi}^{-1} o \eta ) } - \frac{1}{8\pi R^2 }, $$

 indeed, we have

 $$ \left\{ \begin{array}{lllll}
 \displaystyle \Delta_{\mathbb{H}^1} \left( \frac{1}{8\pi R^2 }\right)  & = & 0 & \mbox{in} & B(\xi, R)  \\
 \displaystyle \frac{1}{8\pi R^2 } & = & \displaystyle \frac{1}{8\pi \rho^2 ( {\xi}^{-1} o \eta ) } & \mbox{on} & \partial B(\xi, R) .
 \end{array} \right. $$

 \end{example}\begin{flushright}
    $\square$
\end{flushright}

 \begin{example} The Green's function of the open set $D= \lbrace \eta(x, y, t) \in \mathbb{H}^1, t>0 \rbrace $, with pole at
  $ \xi = (0, 0, t_{0}) ,$ recall that the boundary of $D$ is  $\partial D=\partial \lbrace \eta(x, y, t) \in \mathbb{H}^1, t=0
  \rbrace.$ We have

 $$    G_{D,\xi}(\eta) = \frac{1}{8\pi \rho^2 ( {\xi}^{-1} o \eta ) } - \frac{1}{8\pi \rho^2 ( {\xi} o \eta ) }, $$

 in fact

 $$ \left\{ \begin{array}{lllll}
 \displaystyle \Delta_{\mathbb{H}^1} \left( \frac{1}{8\pi \rho^2 ( {\xi} o \eta ) }\right)   & = & 0 & \mbox{in} & D \\
 \displaystyle \frac{1}{8\pi \rho^2 ( {\xi} o \eta ) } & = & \displaystyle \frac{1}{8\pi \rho^2 ( {\xi}^{-1} o \eta ) } & \mbox{on} &
 \partial D
 \end{array} \right. $$\begin{flushright}
    $\square$
\end{flushright}

 \end{example}

 \subsection{Properties of the  Green's function}

 Let $ \Omega $ be a regular domain of $ \mathbb{H}^1 $ and $ G_{\xi} $  the Green's function
  associated to $ \Omega $  with pole at the point $\xi.$

 \begin{property}

 The function $ G_{\xi} $ is positive on $ \Omega $.

 \begin{proof}

 We have $ \displaystyle \lim_{\eta \rightarrow \xi} G_{\xi}(\eta) = +\infty $, so $ \forall M > 0 $, there exists $ r > 0 $ such that

  \begin{equation} \label{positive1} G_{\xi}(\eta) > M > 0,\, \forall \eta \in \overline{B(\xi, r)} \end{equation}.

 Let  $ D = \Omega \backslash \overline{B(\xi, r)}, $  the function  $ G_{\xi} $ is harmonic on $ D $. As $ G_{\xi} $ vanishes on
  $ \partial \Omega $ and is strictly positive on $ \partial \overline{B(\xi, r)}$, thus by using the maximum principle given in
   Theorem \ref{maximumprinciple}, we obtain

 \begin{equation} \label{positive2} G_{\xi} \geq 0  \ \mbox{on} \ D. \end{equation}

 Through ($\ref{positive1}$) and ($\ref{positive2}$), we conclude that $  G_{\xi} \geq 0 $ on $ \Omega $.

 \end{proof}

 \end{property}

 \begin{property}

 The Green's function is symmetric, it is a consequence of the fact that the Kohn-Laplace operator is a self
 adjoint operator.

 \end{property}\begin{flushright}
    $\square$
\end{flushright}

 \begin{property}

 If $ \Omega_{1} \subset \Omega_{2} $, we have  for all $\xi \in \Omega_{1} $

 $$ \displaystyle G_{\left\{\Omega_{1}; \xi \right\}} \leq \, G_{\left\{\Omega_{2}; \xi
 \right\}}\,
 \text{on}\,\,\, \Omega_{1}  $$

 \end{property}\begin{flushright}
    $\square$
\end{flushright}

 \begin{proof}

  We consider for $ \xi \in \Omega_{1} , $ the function defined on $\Omega_{1} $

 $$ \displaystyle \psi_{\xi} = G_{\lbrace \Omega_{2}, \xi \rbrace } - G_{\lbrace \Omega_{1},
 \xi \rbrace } .$$

 so

 $$ \displaystyle \psi_{\xi} = H_{\lbrace \Omega_{1}, \xi \rbrace} - H_{\lbrace \Omega_{2}, \xi \rbrace}, $$

 thus, $ \psi_{\xi} $ is harmonic on  $ \Omega_{1} $. On the other hand,  as $ G_{\lbrace \Omega_{2}, \xi \rbrace } $ is positive on $ \Omega_{2} $, then  $ \forall \eta \in \partial \Omega_{1} $,

 $$ \displaystyle \psi_{\xi} = G_{\lbrace \Omega_{2}, \xi \rbrace } \geq 0, $$

 by using the maximum principle,  we have $ \forall \eta \in \Omega_{1} $

 $$ \displaystyle  \psi_{\xi} \geq 0. $$

 Therefore

 \begin{equation} \label{greenmonotone} G_{\lbrace \Omega_{2}, \xi \rbrace } \geq G_{\lbrace \Omega_{1}, \xi \rbrace }. \end{equation}

 \end{proof}

  For any real $ c ,$ we denote by $\Omega(c)= \{ \eta \in \Omega,\, G_{\xi}( \eta) > c \} ,$ the level set of the Green's function
  on $\Omega$  and by  $\partial \Omega(c)= \{ \eta \in \Omega,\, G_{\xi}( \eta) = c \}$ its
  boundary.

 \begin{theorem}

  \label{greentheorem}  Let $ G_{\xi} $ be the Green's function of $ \Omega $ with pole at $ \xi $. For almost every
  $\tau,$  it holds

  $$ \int_{\partial \Omega(\tau)} \frac{\left| \nabla_{\mathbb{H}^1}G_{\xi} \right|^2}{\left| \nabla G_{\xi} \right|} d \Sigma_{\Omega(\tau)} = 1. $$\begin{flushright}
    $\square$
\end{flushright}

  \end{theorem}

  \begin{proof}

  Let $ D = \Omega(\tau) \backslash \overline{B(\xi, \epsilon )}$. So $ G_{\xi} = \Gamma_{\xi} - H_{\xi} $ is harmonic on $ D $ and we have

  $$ \int_{D} \Delta_{\mathbb{H}^1} G_{\xi} dv = 0 . $$

  An integration by parts gives

  $$ \begin{array}{lll}
  \displaystyle\int_{D} \Delta_{\mathbb{H}^1} G_{\xi} dv & = & \displaystyle\int_{\partial \Omega(\tau)} \left\langle \nabla_{\mathbb{H}^1} G_{\xi}, N_{1} \right\rangle d\Sigma_{\Omega(\tau)} - \disp\int_{ \partial B(\xi, \epsilon ) } \left\langle \nabla_{\mathbb{H}^1} G_{\xi}, N_{2} \right\rangle d\Sigma_{B(\xi, \epsilon )} \\

  & = & 0 ,
  \end{array} $$

  where

  $$ N_{1} = -\frac{\nabla G_{\xi}}{ \left| \nabla G_{\xi} \right| } $$

  and

  $$ N_{2} = -\frac{ \nabla \rho( \xi^{-1} o \eta ) }{ \left| \nabla \rho( \xi^{-1} o \eta ) \right| }. $$

  Thus

  $$ \int_{ \partial \Omega(\tau) } \left\langle \nabla_{\mathbb{H}^1} G_{\xi}, N_{1} \right\rangle d\Sigma_{\Omega(\tau)} = \int_{ \partial B(\xi, \epsilon ) } \left\langle \nabla_{\mathbb{H}^1} G_{\xi}, N_{2} \right\rangle d\Sigma_{B(\xi, \epsilon )}. $$

  On one hand

  $$ \int_{ \partial B(\xi, \epsilon ) } \left\langle \nabla_{\mathbb{H}^1} G_{\xi}, N_{2} \right\rangle d\Sigma_{B(\xi, \epsilon )} = \int_{ \partial B(\xi, \epsilon ) } \left\langle \nabla_{\mathbb{H}^1} \Gamma_{\xi}, N_{2} \right\rangle d\Sigma_{B(\xi, \epsilon )} - \int_{ \partial B(\xi, \epsilon ) } \left\langle \nabla_{\mathbb{H}^1} H_{\xi}, N_{2} \right\rangle d\Sigma_{B(\xi, \epsilon )} , $$

  then, by proceeding with  an integration by parts in the second
  integral of the equality above, we obtain

  $$ \int_{ \partial B(\xi, \epsilon ) } \left\langle \nabla_{\mathbb{H}^1} H_{\xi}, N_{2} \right\rangle d\Sigma_{B(\xi, \epsilon )} =
  \int_{B(\xi, \epsilon )} \Delta_{\mathbb{H}^1} H_{\xi} dv = 0.$$

 On the other hand

  $$ \int_{ \partial B(\xi, \epsilon ) } \left\langle \nabla_{\mathbb{H}^1} \Gamma_{\xi}, N_{2} \right\rangle d\Sigma_{B(\xi, \epsilon )} = -1, $$

 therefore

  $$ \int_{\partial \Omega(\tau)} \left\langle \nabla_{\mathbb{H}^1} G_{\xi}, N_{1} \right\rangle d\Sigma_{\Omega(\tau)} = -1. $$

  The result follows.

  \end{proof}

 \section{ The harmonic radius in the Heisenberg group }

 \subsection{ Definition and examples }

 The harmonic radius of a regular domain $ \Omega \subset \mathbb{H}^1 $ at a fixed point $ \xi $, denoted by $ r_{\Omega}(\xi) $ is  defined by

 $$ \Gamma (r_{\Omega}(\xi)) = H_{\left\{\Omega; \xi \right\}}(\xi)$$

where

 $$ \Gamma(r) = \displaystyle \frac{1}{8\pi r^2} $$

 and $ H_{\left\{\Omega; \xi \right\}}$ is the regular part of the Green's function $ G_{\left\{\Omega; \xi \right\}}$. \\
Using the expression of the function $ \Gamma,$ we obtain
$$r_{\Omega}(\xi)=\frac{1}{\sqrt{8\pi H_{\left\{\Omega; \xi \right\}}(\xi) }} $$\begin{flushright}
    $\square$
\end{flushright}

 \begin{Remarque}

 As $ H_{\left\{\Omega; \xi \right\}} $ is strictly positive, the harmonic radius is well defined.

 \end{Remarque}

The difficulty to explicit the harmonic radius at all points of any
domain  $\Omega $ is due to the difficulty to explicit Green's
functions for such domain as we noticed it earlier. So as examples,
we will come back to the ones given in the last section.

 \begin{example} \label{exampleharmonicradius}

  Let $ \Omega $ be the Koranyi's ball $ B(\xi, R) $. The associated Green's function with
  pole at the center $\xi$ of the ball, is

 $$ G_{\left\{\Omega; \xi \right\}}(\eta) = \frac{1}{8\pi \rho^2 ( {\xi}^{-1} o \eta ) } - \frac{1}{8\pi R^2 }, $$

 so for all $\eta \in B(\xi, R) ,$  $ H_{\left\{\Omega; \xi \right\}}(\eta) = \displaystyle \frac{1}{8\pi R^2 }. $

 Hence, we deduce that the harmonic radius at the center of the ball is
 equal to its radius

 $$ r_{\Omega}(\xi) = R. $$\begin{flushright}
    $\square$
\end{flushright}

 \end{example}

 \begin{example}

 Let $ D = \lbrace \eta(x, y, t) \in \mathbb{H}^1,\, t>0 \rbrace $. For $ \xi = (0, 0, t_{0}) ,$ we have

 $$   G_{\left\{\Omega; \xi \right\}}(\eta) = \frac{1}{8\pi \rho^2 ( {\xi}^{-1} o \eta ) } - \frac{1}{8\pi \rho^2 ( {\xi} o \eta ) }, $$

 so

 $$ H_{\left\{\Omega; \xi \right\}}(\eta) = \frac{1}{8\pi \rho^2 ( {\xi} o \eta ) }, $$

 and

 $$ H_{\left\{\Omega; \xi \right\}}(\xi) = \frac{1}{8\pi \rho^2 ( {\xi} o \xi ) } = \frac{1}{16 \pi t_{0}}, $$

 then

 $$ r_{\Omega}(\xi) = \sqrt{2t_{0}}. $$\begin{flushright}
    $\square$
\end{flushright}

 \end{example}

 \subsection{Properties of The harmonic radius }

 In the sequel and for the sake of simplicity, the functions $ G_{\left\{\Omega; \xi \right\}} $ and
  $  H_{\left\{\Omega; \xi \right\}} $ will be respectively denoted by $ G_{\xi} $ and $ H_{\xi}.$

 \begin{property}
  For any regular domain  $ \Omega $  of $ \mathbb{H}^1 $  and any  $ \xi \in \Omega $, we have

  \begin{equation} \label{infrsup} \inf_{\eta \in \partial\Omega} \rho ( {\xi}^{-1} o \eta ) \leq r_{\Omega}(\xi) \leq \sup_{\eta \in \partial \Omega} \rho ( {\xi}^{-1} o \eta ) \end{equation}

 \end{property}\begin{flushright}
    $\square$
\end{flushright}

\begin{proof}.
 Since $ H_{\xi} $ is harmonic on $ \Omega $, so by using the maximum principle for the Kohn-Laplace operator, we obtain

 \begin{equation} \label{infH} \inf_{\eta \in \Omega} H_{\xi}(\eta) = \inf_{\eta \in \partial \Omega} H_{\xi}(\eta),\end{equation}

 and

 \begin{equation}  \label{supH} \sup_{\eta \in \Omega} H_{\xi}(\eta) = \sup_{\eta \in \partial \Omega} H_{\xi}(\eta). \end{equation}

 On the other hand, $ \forall \eta \in \partial \Omega $

 \begin{equation} \label{H} H_{\xi}(\eta) = \frac{1}{8\pi \rho^2 ( {\xi}^{-1} o \eta ) }, \end{equation}

 thus, by using ($ \ref{infH} $), ($ \ref{supH} $) and ($ \ref{H} $), we obtain

 \begin{equation}\label{infHsup}  \inf_{\eta \in \partial \Omega} \frac{1}{8\pi \rho^2 ( {\xi}^{-1} o \eta ) }  \leq H_{\xi}(\eta) \leq \sup_{\eta \in \partial \Omega} \frac{1}{8\pi \rho^2 ( {\xi}^{-1} o \eta ) } \end{equation}

The result follows using the definition of the harmonic radius.

 \end{proof}

 \begin{property}

 The harmonic radius is monotone, it means that, if $ \Omega_{1} \subset \,\Omega_{2} ,$ then for all $\xi $ in $\Omega_{1}$

 $$ r_{\Omega_{1}}(\xi) \leq \, r_{\Omega_{2}}(\xi) $$\begin{flushright}
    $\square$
\end{flushright}

 \end{property}

 \begin{proof}

 We consider the function

 $$ \psi_{\xi} = G_{\lbrace \Omega_{2}, \xi \rbrace } - G_{\lbrace \Omega_{1}, \xi \rbrace } $$

 we have

 $$ \psi_{\xi} = H_{\lbrace \Omega_{1}, \xi \rbrace} - H_{\lbrace \Omega_{2}, \xi \rbrace}. $$

 As $ \psi_{\xi} $ is positive,  we obtain

 $$ H_{\lbrace \Omega_{1}, \xi \rbrace}(\xi) \geq H_{\lbrace \Omega_{2}, \xi \rbrace}(\xi).$$

 The result follows.

 \end{proof}

  \begin{property} \textbf{Boundary Behavior}\\
Since the regular part of the Green's function coincides with its
singular part on the boundary of the domain, the harmonic radius
vanishes at every point of the boundary. More precisely, Let $\xi,
\eta \in \partial\Omega,$ then $$\lim_{\eta\rightarrow \xi}
H_{\xi}(\eta)= \lim_{\eta\rightarrow \xi}
\Gamma_{\xi}(\eta)=\infty.$$ Hence $r_{\Omega}(\xi)=0.$\\
Furthermore, we have the following expansion of the Harmonic radius
near the boundary of domains $\Omega$ such that $\Omega$ and
$\mathbb{H}^{1}\setminus \overline{\Omega}$ satisfy
$(\ref{boulexter}).$ For a point $\xi \in \Omega,$ let us denote
$d_{\xi}=d(\xi,\partial(\Omega))$ its distance  from the boundary of
$\Omega,$ then
$$H_{\xi}(\xi)= \frac{1}{(2 d_{\xi})^{2}}+ o(d_{\xi}^{-2})$$
For a proof one can see $\cite{16}.$ Hence, the harmonic radius has
the following expansion near the boundary
$$ r_{\Omega}(\xi)=2 d_{\xi}+ o(d_{\xi})$$\begin{flushright}
    $\square$
\end{flushright}

  \end{property}
In the following, we will give the relation between the harmonic
radius of $\Omega$ and those of its level sets.
 \begin{proposition}

 Let $ \Omega(\tau) = \lbrace \eta \in \Omega,\, G_{\xi}(\eta) > \tau \rbrace $ with $ \tau $  a positive real and $\xi \in \Omega(\tau).$
 We have the following relation between the harmonic radius of  $ \Omega(\tau) $ at  $\xi$ denoted by $ r_{\Omega(\tau)}(\xi) $ and the one of
  $\Omega$ at the same point

 \begin{equation} \label{rtau} r_{\Omega(\tau)}(\xi) = \frac{r_{\Omega}(\xi)}{\sqrt{ \tau 8\pi r_{\Omega}^{2}(\xi) + 1}} \end{equation}\begin{flushright}
    $\square$
\end{flushright}

 \end{proposition}

 \begin{proof}

 We define on $ \Omega(\tau),$ the function

 $$ G_{\xi}^{\tau}(\eta) = G_{\xi} - \tau = \frac{1}{8\pi \rho^2 ( {\xi}^{-1} o \eta ) } - H_{\xi}(\eta) - \tau. $$

 As $ G_{\xi}^{\tau} $ vanishes on $ \partial \Omega(\tau) $ and the function $ H_{\xi}^{\tau} = H_{\xi} + \tau $ is harmonic on $ \Omega(\tau) ,$  we conclude that $ G_{\xi}^{\tau} $ is the Green's function associated to the set $ \Omega(\tau) $ and $ H_{\xi}^{\tau} = H_{\xi} + \tau $ is its regular part, thus the harmonic radius of $ \Omega(\tau) $ at $\xi$, is given by

 $$ \begin{array}{lll}

 r_{\Omega(\tau)}(\xi) & = & \displaystyle \frac{1}{ \sqrt{8\pi H_{\xi}^{\tau}(\xi)}} \\

 & = & \displaystyle \frac{1}{\sqrt{8\pi \left( H_{\xi}(\xi) + \tau \right)}} \\

 & = & \displaystyle \frac{1}{\sqrt{8\pi \left( \displaystyle \frac{1}{ 8\pi r_{\Omega}^{2}(\xi)} + \tau \right)}} \\

 & = & \displaystyle \frac{r_{\Omega}(\xi)}{\sqrt{\tau 8\pi r_{\Omega}^{2}(\xi) + 1} }.

 \end{array} $$

 \end{proof}

Next, we will give the relation between the harmonic radius of
$\Omega$ at a given point $\xi \in \Omega$ and the volume of
$\Omega.$

  \begin{theorem}

 Let $ \Omega $ be a  domain of $ \mathbb{H}^1 $ such that $\Omega$ and $\mathbb{H}^{1}\setminus \bar{\Omega}$ satisfy the uniform exterior property (\ref{boulexter})and let $\xi$ be a point of  $ \Omega.$ We have the following inequality

 \begin{equation} \label{rR}  r_{\Omega}(\xi) \leq \,\alpha_{\Omega}(\xi)\, (\frac{2|\Omega|}{\pi^{2}})^{\frac{1}{4}}, \end{equation}

 with

  $$ \alpha_{\Omega}(\xi) =    \frac{ \displaystyle \sup_{\eta \in \partial \Omega} d_{\mathbb{H}^1}(\xi, \eta)}{\displaystyle \inf_{ \eta \in \partial \Omega} d_{\mathbb{H}^1}(\xi, \eta)} $$\begin{flushright}
    $\square$
\end{flushright}

 \end{theorem}

 \begin{proof}

 Since $\Omega$ and $\mathbb{H}^{1}\setminus \bar{\Omega}$ satisfy  (\ref{boulexter}) , the ball $ B(\xi, \displaystyle \inf_{\eta \in \partial \Omega} d_{\mathbb{H}^1}(\xi, \eta) ) \subset \Omega, $ hence

 $$ \vert B(\xi, \displaystyle \inf_{\eta \in \partial \Omega} d_{\mathbb{H}^1}(\xi, \eta) )\vert \leq \vert \Omega \vert, $$

 thus

 $$ \frac{\pi^2}{2} \left( \inf_{ \eta \in \partial \Omega} d_{\mathbb{H}^1}(\xi, \eta) \right)^4 \leq\, |\Omega|, $$

 which implies

 \begin{equation} \label{infdh} \inf_{ \eta \in \partial \Omega} d_{\mathbb{H}^1}(\xi, \eta) \leq \,\displaystyle{(\frac{2|\Omega|}{\pi^{2}})^{\frac{1}{4}}}. \end{equation}

 The result follows using ($\ref{infrsup}$) and ($\ref{infdh}$).

We derive the following result for the Koranyi's ball of center $0$
and radius $r_{\Omega}(\xi)$
 \begin{corollary} For any $ \xi \in \Omega ,$ we have

 \begin{equation} \label{Balpha}  \vert B(\xi, r_{\Omega}(\xi)) \vert \leq \alpha_{\Omega}^{4}(\xi)\, \vert \Omega \vert \end{equation}\begin{flushright}
    $\square$
\end{flushright}

 \begin{proof}

 The result is a direct consequence of ($\ref{ballvolume}$) and ($\ref{rR}$).
 \end{proof}
 \end{corollary}

 \end{proof}

 \begin{example} \label{examplealpha}

 We come back to Example $\ref{exampleharmonicradius}$, where, we showed that  the harmonic radius of the Koranyi's ball $ B(\xi, r) $ at $ \xi $ is equal to the radius $ r $. Since
 $$ \sup_{ \eta \in \partial B(\xi, r)} d_{\mathbb{H}^1}(\xi, \eta) = \inf_{ \eta \in \partial B(\xi, r)} d_{\mathbb{H}^1}(\xi, \eta) = r, $$

  we derive that

 $$ \alpha_{B(\xi, r)}(\xi) = 1. $$

 Then, using ($\ref{ballvolume}$), we obtain the equality case in ($\ref{rR}$) and ($\ref{Balpha}$).\begin{flushright}
    $\square$
\end{flushright}

 \end{example}

 \begin{proposition}

 For almost every positive real $ \tau ,$ we have

 \begin{equation} \label{alpha} \alpha_{\Omega(\tau)}(\xi) \leq \alpha_{\Omega}(\xi), \end{equation}\begin{flushright}
    $\square$
\end{flushright}

 \end{proposition}

 \begin{proof}
 The Green's function of $ \Omega(\tau) $ is given by

 $$ G_{\Omega(\tau), \xi}(\eta) =  \frac{1}{8\pi \rho^2 ( {\xi}^{-1} o \eta ) } - H_{\xi}(\eta) - \tau, $$

 so, for all $ \eta_{\tau} \in \partial \Omega(\tau) $

 $$ \frac{1}{8\pi \rho^2 ( {\xi}^{-1} o \eta_{\tau} ) } = H_{\xi}(\eta_{\tau}) + \tau. $$

 Next, we use ($\ref{infHsup}$), to obtain

 $$ \frac{1}{8\pi \left( \displaystyle \sup_{ \eta \in \partial \Omega} d_{\mathbb{H}^1}(\xi, \eta)\right)^2 } + \tau \leq \frac{1}{8\pi \rho^2 ( {\xi}^{-1} o \eta_{\tau} ) } \leq \frac{1}{8\pi \left( \displaystyle \inf_{ \eta \in \partial \Omega} d_{\mathbb{H}^1}(\xi, \eta)\right)^2 } + \tau, $$

hence

$$ \frac{ \left( \displaystyle \sup_{ \eta \in \partial \Omega} d_{\mathbb{H}^1}(\xi, \eta)\right)^2}{1+\tau \,8\pi \left( \displaystyle \sup_{ \eta \in \partial \Omega} d_{\mathbb{H}^1}(\xi, \eta)\right)^2} \geq \rho^2 ( {\xi}^{-1} o \eta_{\tau} ) \geq \frac{ \left( \displaystyle \inf_{ \eta \in \partial \Omega} d_{\mathbb{H}^1}(\xi, \eta)\right)^2}{1+\tau\,8\pi \left(\displaystyle \inf_{ \eta \in \partial \Omega} d_{\mathbb{H}^1}(\xi, \eta)\right)^2}. $$

Therefore, we obtain

\begin{equation} \label{suptau} \frac{ \left( \displaystyle \sup_{ \eta \in \partial \Omega} d_{\mathbb{H}^1}(\xi, \eta)\right)^2}{1+\tau\, 8\pi \left( \displaystyle \sup_{ \eta \in \partial \Omega} d_{\mathbb{H}^1}(\xi, \eta)\right)^2} \geq \displaystyle \sup_{ \eta_{\tau} \in \partial \Omega(\tau)} \rho^2 ( {\xi}^{-1} o \eta_{\tau}) = \left( \displaystyle \sup_{ \eta_{\tau} \in \partial \Omega(\tau)} d_{\mathbb{H}^1}(\xi, \eta_{\tau} )\right)^2  \end{equation}

and

\begin{equation} \label{inftau} \left( \displaystyle \inf_{ \eta_{\tau} \in \partial \Omega(\tau)} d_{\mathbb{H}^1}(\xi, \eta_{\tau} )\right)^2 = \displaystyle \inf_{ \eta_{\tau} \in \partial \Omega(\tau)} \rho^2 ( {\xi}^{-1} o \eta_{\tau} ) \geq \frac{ \left(\displaystyle \inf_{ \eta \in \partial \Omega} d_{\mathbb{H}^1}(\xi, \eta)\right)^2}{1+\tau \,8\pi \left( \displaystyle \inf_{ \eta \in \partial \Omega} d_{\mathbb{H}^1}(\xi, \eta)\right)^2}. \end{equation}

By using ($\ref{suptau}$) and ($\ref{inftau}$), we obtain

$$ \alpha_{\Omega(\tau)}^{2}(\xi) \leq \alpha_{\Omega}^{2}(\xi) \frac{1+8\pi\tau \left( \displaystyle \inf_{\eta \in \partial \Omega} d_{\mathbb{H}^1}(\xi, \eta)\right)^2}{1+8\pi\tau \left( \displaystyle \sup_{ \eta \in \partial \Omega} d_{\mathbb{H}^1}(\xi, \eta)\right)^2} \leq \alpha_{\Omega}^{2}(\xi). $$

 \end{proof}

As a consequence of the proposition above, we obtain

 \begin{corollary} For almost every positive real $\tau,$

 \begin{equation} \label{importantcorollary} \left| B(0, r_{\Omega(\tau)}(\xi) ) \right| \leq  \alpha_{\Omega}^{4}(\xi) \left| \Omega(\tau) \right| \end{equation}

 \end{corollary}\begin{flushright}
    $\square$
\end{flushright}

\begin{proof}

We apply ($\ref{Balpha}$) to the domain $ \Omega(\tau) ,$  then we
use ($\ref{alpha}$),
  the result follows.

  \end{proof}

  \section{ Harmonic Transplantation}

  \subsection{ A result of P.Pansu in $ \mathbb{H}^1 $  }

    We have the following result

  If $ \Omega \subset \mathbb{H}^1 $ is a $ C^1 $ bounded domain, then $ P_{\mathbb{H}^1} ( \Omega ) ,$ the horizontal perimeter of $ \Omega $ coincides with the area of
   $ \partial \Omega .$

 \begin{equation} \label{areaomega} A(\partial \Omega) = \int_{\partial \Omega} \left| N_{ \mathbb{H}^1 } \right| d\Sigma_{\Omega}, \end{equation}

 where $ N_{ H } = N - < N, T> T $ denotes the orthogonal projection of $ N $ onto the horizontal distribution, and $ d\Sigma $
 is the Riemannian area element induced by the metric $ g_{\mathbb{H}^1} $ on $ \partial \Omega $ (one can see $ \cite{14, 15, 23} $).\begin{flushright}
    $\square$
\end{flushright}

   Recall that $ P_{\mathbb{H}^1} ( \Omega )$
   is defined by

 $$ P_{\mathbb{H}^1} ( \Omega ) = \sup \left\{ \int_{ \Omega } div V dv ; \ V \mbox{ horizontal of classe} C^1, \ \left| V \right| \leq
 1 \right\}, $$

  where $ div V $ is the Riemannian divergence of the vector field $ V $, and $ dv $ is the volume element associated to the left-invariant
   Riemannian metric $g_{\mathbb{H}^1}.$

  As a consequence, one obtain

 \begin{corollary}{($\cite{14, 15}$)} \label{phd}

 For a set $ D $ of $\mathbb{H}^1$ bounded by a surface $ \lbrace u = 0 \rbrace $ of class $ C^1 $

 $$  P_{\mathbb{H}^1} ( D ) = \int_{\lbrace u = 0 \rbrace} \frac{ \left| \nabla_{\mathbb{H}^1}u \right|}{ \left| \nabla u \right|}
  d\Sigma_{D}. $$\begin{flushright}
    $\square$
\end{flushright}

 \end{corollary}

We have the following result of P.Pansu in $ \cite{6, 7} $
 \begin{theorem}\label{Pansu's theorem}

 Let $ \Omega \subset \mathbb{H}^1 $ be  a $ C^1 $ bounded domain,
 then

 $$  \left| \Omega \right|^{\frac{3}{4}} \leq   C P_{\mathbb{H}^1} ( \Omega ) $$

    with  $ C = \displaystyle (\frac{3}{2\pi})^{\frac{1}{4}}$\begin{flushright}
    $\square$
\end{flushright}

 \end{theorem}

  \subsection{Harmonic Transplantation}

\begin{definition}
 Let $ \Omega $  be a bounded regular subset of $\mathbb{H}^{1}$ satisfying $(\ref{boulexter})$ and $ \xi $  a point  in $\Omega.$
  If $ u : B(0, r_{\Omega}(\xi))\rightarrow \mathbb{R}_{+} $ is a $\rho$-radial function on the Koranyi's ball of center $0$
   and radius $r_{\Omega}(\xi)$ and $\varphi $  is a function defined by
$ \varphi( G_{\left\{ B(0, r_{\Omega}(\xi)), 0 \right\}})=u. $ Then
the harmonic transplantation $ U_{\xi} $ of $ u $ into
 $ \Omega $ is the function
$$ U_{\xi} = \varphi (G_{\xi}) .$$\begin{flushright}
    $\square$
\end{flushright}
\end{definition}

  \begin{proposition} \label{gradconserve}

  Let $ \Omega $ be a regular bounded domain of $ \mathbb{H}^1 $ and $ \xi $ any point in $ \Omega $, then

   $$  \int_{\Omega} \left| \nabla_{\mathbb{H}^1} U_{\xi} \right|^2 \ dv = \int_{B(0, r_{\Omega}(\xi))} \left| \nabla_{\mathbb{H}^1}u \right|^2 \ dv. $$
\begin{flushright}
    $\square$
\end{flushright}
 \end{proposition}

 \begin{proof}
 Using the co-area formula given in $\cite{18, 19}$, we obtain

 $$ \displaystyle \int_{\Omega} \left| \nabla_{\mathbb{H}^1} U_{\xi} \right|^2 \ dv  =  \displaystyle \int_{0}^{\infty} \left( \int_{ \left\{ G_{\xi} = \tau \right\} } \displaystyle \frac{ \left\langle \nabla_{\mathbb{H}^1} U_{\xi}, \nabla_{\mathbb{H}^1}U_{\xi} \right\rangle }{ \left| \nabla G_{\xi} \right|} \ d\Sigma_{\Omega(\tau)}  \right) \ d \tau $$

 On the other hand

 $$ \nabla_{\mathbb{H}^1}U_{\xi} = \mu^{'} ( G_{\xi} ) \nabla_{\mathbb{H}^1}G_{\xi}. $$

Since

 $$ \begin{array}{lll}

 \displaystyle \int_{\Omega} \left| \nabla_{\mathbb{H}^1} U_{\xi} \right|^2 \ dv & = & \displaystyle \int_{0}^{\infty} \left( \int_{ \left\{ G_{\xi} = \tau \right\} } \displaystyle  \frac{ \left| \mu^{'} ( G_{\xi} ) \right|^2  \left\langle \nabla_{\mathbb{H}^1}G_{\xi}, \nabla_{\mathbb{H}^1}G_{\xi} \right\rangle }{ \left| \nabla G_{\xi} \right|} \ d\Sigma_{\Omega(\tau)}  \right) \ d \tau \\ \\

  & = & \displaystyle \int_{0}^{\infty}  \left| \mu^{'}( \tau) \right|^2 \left( \displaystyle\int_{\partial \Omega(\tau)} \frac{\left| \nabla_{\mathbb{H}^1}G_{\xi} \right|^2}{\left| \nabla G_{\xi} \right|} d \Sigma_{\Omega(\tau)} \right)  \ d
  \tau.

 \end{array} $$

 Therefore, using  Theorem $\ref{greentheorem},$ we obtain

  $$ \int_{\Omega} \left| \nabla_{\mathbb{H}^1} U_{\xi} \right|^2 \ dv = \int_{0}^{\infty}  \left| \mu^{'}( \tau) \right|^2 \ d\tau. $$

  We remark that the integral above is independent of $ \Omega $, thus by the same way, we prove that

$$ \int_{B(0, r_{\Omega}(\xi))} \left| \nabla_{\mathbb{H}^1}u \right|^2 \ dv = \int_{0}^{\infty}  \left| \mu^{'}( \tau) \right|^2 \ d\tau , $$

which completes the proof of the proposition.

\end{proof}

  \begin{proposition}\label{thr}

Let  $ \Omega $ be a regular bounded domain  of $\mathbb{H}^{1}$
satisfying $(\ref{boulexter})$ and $ \xi $  any point in $\Omega$.
 For any  positive continuous  function $f ,$  we have

  $$     \int_{B(0, r_{\Omega}(\xi)) } f(u) \ dv \leq C_{\Omega}(\xi) \int_{\Omega} f(U_{\xi}) \ dv $$
  where $ C_{\Omega}(\xi)= 16 \sqrt{2}\ \alpha_{\Omega}^{6}(\xi) C^{2} $
\begin{flushright}
    $\square$
\end{flushright}
  \end{proposition}

 \begin{proof}
  We apply the co-area formula, we obtain

  $$ \begin{array}{lll}

  \displaystyle \int_{\Omega} f ( U_{\xi}(\eta)) \ dv  & = & \displaystyle \int_{0}^{\infty} \left( \int_{ \left\{ G_{\xi} = \tau \right\} } \frac{ f ( \varphi ( G_{\xi} ( \eta) )) }{ \left| \nabla G_{\xi } (\eta) \right| } d\Sigma_{\Omega(\tau)} \right) d\tau  \\ \\

 & = & \displaystyle \int_{0}^{\infty} f ( \varphi(\tau) ) \left( \int_{ \left\{ G_{\xi} = \tau \right\} } \frac{ 1 }{ \left| \nabla G_{\xi } (\eta) \right| } d\Sigma_{\Omega(\tau)} \right) d\tau.

   \end{array} $$

On one hand,

 $$ \begin{array}{lll}

  \displaystyle \int_{ \left\{ G_{\xi} = \tau \right\} } \frac{ 1 }{ \left| \nabla G_{\xi } (\eta) \right| } d\Sigma_{\Omega(\tau)} & = & \displaystyle \int_{ \left\{ G_{\xi} = \tau \right\} } \frac{ \left| \nabla_{\mathbb{H}^1}G_{\xi}(\eta)  \right|^2 }{ \left| \nabla G_{\xi } (\eta) \right| } d\Sigma_{\Omega(\tau)} . \displaystyle \int_{ \left\{ G_{\xi} = \tau \right\} } \frac{ 1 }{ \left| \nabla G_{\xi } (\eta) \right| } d\Sigma_{\Omega(\tau)} \\ \\

  & \geq & \displaystyle \left( \int_{ \left\{ G_{\xi} = \tau \right\} } \frac{ \left| \nabla_{\mathbb{H}^1}G_{\xi}(\eta)  \right| }{ \left| \nabla G_{\xi } (\eta) \right| } d\Sigma_{\Omega(\tau)} \right)^2 \\ \\

  & = & \displaystyle \left( P_{\mathbb{H}^1}( \left\{ G_{\xi} >\, \tau \right\} ) \right)^2

\end{array} $$

We apply Theorem \ref{Pansu's theorem} to obtain
  $$\displaystyle \int_{ \left\{ G_{\xi} = \tau \right\} } \frac{ 1 }{ \left| \nabla G_{\xi } (\eta) \right| } d\Sigma_{\Omega(\tau)}  \geq  \displaystyle \frac{1}{C^2} \left| \left\{ G_{\xi} >\, \tau \right\}
  \right|^{\frac{3}{2}}$$

  On the other hand, $ \displaystyle G_{\left\{ B(0, r_{\Omega}(\xi)), 0\right\}} = \frac{1}{8\pi \rho^2(\eta)} - \frac{1}{8\pi r^2(\xi)} $,
   so by denoting $ G_{ \left\{ B(0, r_{\Omega}(\xi)), 0\right\}} = \tilde{G}_{0}, $  we have

  $$ \begin{array}{lll}

  \displaystyle \int_{B(0, r_{\Omega}(\xi))} f ( u(\eta)) \ dv & = & \displaystyle \int_{0}^{ +\infty} \left( \int_{ \left\{ \tilde{G}_{0} = \tau \right\} } \frac{ f ( \varphi ( \tilde{G}_{0} ( \eta) ) ) }{ \left| \nabla \tilde{G}_{0} (\eta) \right| } d\Sigma_{ \left\{ \tilde{G}_{0} > \tau \right\} } \right) d\tau \\ \\

  & = & \displaystyle \int_{0}^{ +\infty} f ( \varphi(\tau) ) \left( \int_{ \left\{ \tilde{G}_{0}= \tau \right\} } \frac{ 1 }{ \left| \nabla  \tilde{G}_{0} (\eta) \right| } d\Sigma_{ \left\{ \tilde{G}_{0} > \tau \right\} } \right) d\tau \\ \\

  & = & \displaystyle  \int_{0}^{ +\infty} f ( \varphi(\tau) ) \left(4 \pi (r_{\Omega(\tau)}(\xi))^3 \int_{ \left\{ \rho(\eta) = r_{\Omega(\tau)}(\xi) \right\} } \frac{ 1 }{ \left| \nabla  \rho (\eta) \right| } d\Sigma_{B(0, r_{\Omega(\tau)}(\xi))} \right) d\tau.

 \end{array} $$

 Recall the volume of the Koranyi's ball $B(0, r_{\Omega(\tau)}(\xi))$

 $$ \int_{ B(0, r_{\Omega(\tau)}(\xi)) } dv = \left| B(0, r_{\Omega(\tau)}(\xi) ) \right| = \frac{\pi^2}{2} (r_{\Omega(\tau)}(\xi))^4 .$$

 Applying again the co-area formula, we find

 $$ \int_{ B(0, r_{\Omega(\tau)}(\xi)) } dv = \int_{0}^{ r_{\Omega(\tau)}(\xi)} \left( \int_{ \left\{ \rho(\eta) = r \right\} } \frac{ 1 }{ \left| \nabla  \rho (\eta) \right| } d\Sigma_{B(0, r)} \right) dr $$

 thus

 $$ \int_{ \left\{ \rho(\eta) = r_{\Omega(\tau)}(\xi) \right\} } \frac{ 1 }{ \left| \nabla  \rho (\eta) \right| } d\Sigma_{B(0, r_{\Omega(\tau)}(\xi))} = \frac{ \partial |(B(0, r_{\Omega(\tau)}(\xi)))|} { \partial r_{\Omega(\tau)}(\xi)} = 2 \pi^2 (r_{\Omega(\tau)}(\xi))^3 $$

 which gives  using ($\ref{importantcorollary}$),

 $$ \begin{array}{lll}

  \displaystyle \int_{B(0, r_{\Omega}(\xi))} f ( u(\eta) ) \ dv & = & \displaystyle \int_{0}^{ +\infty} f ( \varphi(\tau) )  8\pi^3 (r_{\Omega(\tau)}(\xi))^6 \ d\tau \\ \\

 & = & \displaystyle  \int_{0}^{ +\infty} f ( \varphi(\tau)) 16\sqrt{2} \left| B(0, r_{\Omega(\tau)}(\xi) ) \right|^{\frac{3}{2}} \ d\tau \\ \\

 & \leq & \displaystyle  \int_{0}^{ +\infty} f ( \varphi(\tau) )\alpha_{\Omega}^{6}(\xi) 16 \sqrt{2}  \left|  \left\{ G_{\xi} > \tau \right\}   \right|^{\frac{3}{2}} \ d\tau \\ \\

 & \leq & \displaystyle \alpha_{\Omega}^{6}(\xi)\, 16 \sqrt{2}\, C^{2}\,  \int_{0}^{ +\infty} f ( \varphi(\tau) ) \left( \int_{ \left\{ G_{\xi} = \tau \right\} } \frac{ 1 }{ \left| \nabla G_{\xi } (\eta) \right| } d\Sigma_{ \left\{ G_{\xi} = \tau \right\} } \right) d\tau \\ \\

 & \leq & C_{\Omega}(\xi) \int_{\Omega} f ( U_{\xi} (\eta) ) \ dv.

 \end{array}  $$
 \end{proof}

 \section{ The first eigenvalue of the Kohn-Laplace operator }
 In this section, we will use  Proposition $ \ref{thr} $  to find an upper bound for the first eigenvalue $\lambda_{1}$
 of a regular bounded domain in the Heisenberg group $ \mathbb{H}^1 $  associated to  the following problem:

 \begin{equation}  (P_{\Omega}) \left\{ \begin{array}{lllll}
      -\Delta_{\mathbb{H}^1} u & = & \lambda u & \mbox{in} & \Omega  \\
     u & = & 0 & \mbox{on} & \partial \Omega  \\
      \end{array} \right.  \end{equation}

 called The Kohn-Laplace problem. We recall that $ \Delta_{\mathbb{H}^1} = X^2 + Y^2 $ is the Kohn-Laplace operator in the Heisenberg group
  $ \mathbb{H}^1 $, where $ \displaystyle X = \frac{ \partial }{ \partial x } + 2y\frac{ \partial }{ \partial t} $ and
  $ \displaystyle Y = \frac{ \partial }{ \partial y } - 2x\frac{ \partial }{ \partial t} $. \\
The existence of the eigenvalues of $(P_{\Omega}) $ was proved  by
X.Luo and P.Niu in $ \cite{10, 11, 12} $ by using the Kohn
inequality ($ \cite{13} $)
for the two vector fields $ { X, Y } $ with the spectral properties of compact operators. \\

The problem $(P_{\Omega})$ does not admit a radial solution in the
ball case, which imposes a constraint for the use of  the method of
  Harmonic transplantation in the order to find   an upper bound for the first eigenvalue as in the Euclidean case.
 Indeed, if $ u $ is $\rho -$radial, so there exists a real function  $ f $ such that $ u(\eta) = f(\rho(\eta)) $, thus

 $$ \Delta_{\mathbb{H}^1} u(\eta) = \frac{ x^2+ y^2 }{ \rho^2(\eta)} \left( f''( \rho(\eta)) + \frac{3}{ \rho(\eta)}f'( \rho(\eta)) \right) $$

 and by using the polar coordinates, we obtain

 $$  \Delta_{\mathbb{H}^1} u(\eta) = \sin(\theta) \left( f''(r) + \frac{3}{ r}f'(r) \right) . $$

 So, to obtain a radial solution of the Kohn-Laplace problem for the Koranyi's ball  $ B(0, R) ,$ we have to solve the following problem :

  $$   \left\{ \begin{array}{lll}
 -\sin(\theta) \left( f''(r) + \frac{3}{r}f'(r) \right) & = & \lambda f(r) \ \ in \ ]0, R[\times ]0, \Pi [ \\
 f(R) & = & 0 \\
 \end{array} \right. $$

 which obviously does not admit a radial solution.

The idea to find an  upper bound for the first eigenvalue of
$(P_{\Omega})$ is to perturb  $(P_{B(0, R)})$   and consider the
following problem

 \begin{equation} \label{eqn2} \left\{ \begin{array}{lllll}
 -\Delta_{\mathbb{H}^1} u & = & \mu(R) \psi u & \mbox{in} & B(0,R)\backslash\{0\} \\
  u & = & 0 &  \mbox{on} & \partial B(0,R)
 \end{array} \right. \end{equation}

 where $ \displaystyle \psi(\eta)= \frac{ x^2+ y^2 }{ \rho^2(\eta)} = \left| \nabla_{\mathbb{H}^1}\rho \right|^2 $. We will prove
  that problem ($\ref{eqn2}$) admits a radial solution given by  Bessel functions. To this aim, let us consider the function

 $$ J_{m}(l) = \left( \frac{l}{2} \right)^m \sum_{k=0}^{\infty} \frac{(-1)^k}{k! \Gamma(m+k+1) } \left( \frac{l}{2} \right)^{2k} $$

 which is the Bessel function of first kind with order $ m,$ it  is a solution of the following equation

 $$ \frac{d^2 y}{dx^2}+ \frac{1}{x} \frac{dy}{dx} + \left( 1- \frac{m^2}{x^2} \right) y = 0, $$

called Bessel equation, we refer to $\cite{20}.$  The function $
J_{m}(l) $ admits an infinity of simple and positive zeros:

 $$ 0 < l_{m1} < l_{m2} < ..., \ J_{m}(l_{mj}) = 0, \ j \in \huge{N^*} $$.

   The function Gamma $ \Gamma $ is defined on $ ]0, \infty[ $, by $ \Gamma(x) = \int_{0}^{+\infty} \exp^{-t} t^{x-1} dt $

 and verifies $ \Gamma(n) = (n-1)! ,$ \,\,$\forall n \in N^{*} .$\\

 To find a radial solution for problem ($\ref{eqn2}$), firstly, we have to solve the following problem

 \begin{equation} \label{fr}  \left\{ \begin{array}{lllll}
   f''(r) + \frac{3}{r}f'(r) + \mu(R) f(r) & = & 0 & on & ]0, R[ \\

 f(R) & = & 0.
 \end{array} \right. \end{equation}

 We consider $ f(r) = \displaystyle \frac{1}{r} g(r) $ with $ r > 0 $, so

 $$ \begin{array}{lll}

 \displaystyle f'(r) & = & \displaystyle -\frac{1}{r^2} g(r) + \frac{1}{r} g'(r) \\ \\

 \displaystyle f''(r) & = & \displaystyle \frac{2}{r^3} g(r) - \frac{2}{r^2} g'(r)+ \frac{1}{r} g''(r)

 \end{array} $$

thus

 $$ f''(r) + \frac{3}{r}f'(r) + \mu(R) f(r) = 0 \Leftrightarrow g''(r) + \frac{1}{r}g'(r) + (\mu(R)-\frac{1}{r^2} ) g(r) = 0 $$

 and performing the change of variables  $ l =  \sqrt{\mu(R)} r ,$ we obtain

 $$  g''(r) + \frac{1}{r}g'(r) + (\mu(R)-\frac{1}{r^2} ) g(r) = 0 \Leftrightarrow g''(l) + \frac{1}{l}g'(l) + (1-\frac{1}{l^2} ) g(l) = 0 .$$

 The equation

 $$  g''(l) + \frac{1}{l}g'(l) + (1-\frac{1}{l^2} ) g(l) = 0 .$$

admits  the Bessel function$ J_{1}(l) $ as  solution. Then $
J_{1}(\sqrt{\mu}\, r) $ is a solution of

  $$  g''(r) + \frac{1}{r}g'(r) + (\mu(R)-\frac{1}{r^2} ) g(r) = 0 $$

  which implies that $ \displaystyle \frac{1}{r}J_{1}(\sqrt{\mu}\, r) $ is a solution of

  $$  f''(r) + \frac{3}{r}f'(r) + \mu(R) f(r) = 0. $$

 So  problem ($\ref{fr}$) admits as eigenvalues and associated eigenfunctions respectively

 $$ \mu_{1j}(R) = \frac{ l_{1j}^2 }{R^2}, \ f_{1j}(r) = \displaystyle \frac{1}{r} J_{1} \left( \frac{l_{1j}r}{R} \right), \ j = 1, 2, ....  $$

where the $ l_{1j}$ 's are the strictly positive zeros of the Bessel function $ J_{1}(l) .$ \\
 Next, we remark that on $B(0, R)\setminus\{0\},$ we have

 $$  \Delta_{\mathbb{H}^1} f_{1j}(\rho(\eta)) + \mu_{1j}(R) \psi f_{1j}(\rho(\eta)) = 0 \,\,  \forall j = 1, 2, ... $$

 Therefore,  the radial eigenfunctions and the associated eigenvalues of  problem ($\ref{fr}$)

   $$ u_{1j}(\eta) = f_{1j}(\rho(\eta)), \ \ \ \mu_{1j}(R) = \frac{ l_{1j}^2 }{R^2}, \,\, \forall j = 1, 2, ... $$

 are respectively solutions of problem $(\ref{eqn2}).$  Hence, we derive the following result

   \begin{corollary} \label{solution gradient} For $ j = 1, 2, ... ,$  we have

   $$ \int_{B(0, R)} \left| \nabla_{\mathbb{H}^1}u_{1j} \right|^2 \ dv = \mu_{1j}(R) \int_{B(0, R)} \psi u_{1j}^2 \ dv $$\begin{flushright}
    $\square$
\end{flushright}

   \end{corollary}
Now, we  are ready to state our main result
   \begin{theorem} \label{eigenvalue}
Let $ \Omega $ be a  bounded domain of $ \mathbb{H}^1 $ with smooth
boundary, such that $\Omega$ and $\mathbb{H}^{1}\setminus
\bar{\Omega}$ satisfy the uniform exterior ball  property
(\ref{boulexter}) and let $\xi$ be any point in  $\Omega,$ we have

   $$ \lambda_{1}( \Omega ) \leq    C_{\Omega}(\xi) \displaystyle \frac{l_{11}^2}{r_{\Omega}^2(\xi)}$$

   where $ l_{11} $ is the first strictly positive zero of the Bessel function $ J_{1}(l)$  and $ r_{\Omega}(\xi) $ is the harmonic radius
   of $ \Omega $ at $ \xi.$\begin{flushright}
    $\square$
\end{flushright}

   \end{theorem}

  \begin{proof}

  Using Rayleigh's principle, the first eigenvalue of problem $(P_{\Omega})$ is
  given by

   $$ \lambda_{1}( \Omega )=\displaystyle  \inf_{ v \in  S^{1,2}_{0}(\Omega) }\displaystyle \frac { \displaystyle\int_{\Omega} \left| \nabla_{\mathbb{H}^1} v \right|^2}{\displaystyle \int_{\Omega} v^2}.  $$

 Hence
 $$ \lambda_{1}( \Omega )\leq \displaystyle \frac{\displaystyle\int_{\Omega} \left| \nabla_{\mathbb{H}^1} U_{\xi, 11} \right|^2}{\displaystyle\int_{\Omega} U_{\xi, 11} ^2}$$

 where $ U_{\xi, 11} $ is the harmonic transplantation of  $ u_{11} $ in  $ \Omega .$

 Using  Propositions $\ref{gradconserve},$ $ \ref{thr}$ and Corollary \ref{solution gradient}, we obtain

  $$ \begin{array}{lll}

   \lambda_{1}( \Omega )& \leq & C_{\Omega}(\xi)  \frac{\displaystyle \int_{B(0, r_{\Omega}(\xi))} \left| \nabla_{\mathbb{H}^1}u_{11} \right|^2 \ dv}{\displaystyle \int_{B(0, r_{\Omega}(\xi))}u_{11}^{2} \ dv} \\ \\

   & \leq & C_{\Omega}(\xi) \frac{ \displaystyle \int_{B(0, r_{\Omega}(\xi))} \left| \nabla_{\mathbb{H}^1}u_{11} \right|^2 \ dv}{\displaystyle \int_{B(0, r_{\Omega}(\xi))} \psi u_{11}^{2} \ dv} \\ \\

   & \leq & C_{\Omega}(\xi) \mu_{11}(r_{\Omega}(\xi)) =  C_{\Omega}(\xi) \displaystyle \frac{l_{11}^2}{r_{\Omega}^2(\xi)}.

   \end{array}  $$\begin{flushright}
    $\square$
\end{flushright}
   \end{proof}

We can refine the result of the theorem above, more precisely we
have

  \begin{theorem}\label{Harmonic center}
Let $ \Omega $ be a  bounded domain of $ \mathbb{H}^1 $ with smooth
boundary, such that $\Omega$ and $\mathbb{H}^{1}\setminus
\bar{\Omega}$ satisfy the uniform exterior ball  property
(\ref{boulexter}) and let $\xi$ be any point in $\Omega,$ we have
   $$ \lambda_{1}( \Omega ) \leq C_{\Omega}\frac{ l_{11}^2 }{ \displaystyle \max_{ \xi \in \Omega } r_{\Omega}^2(\xi)}$$

   with $ C_{\Omega} = \displaystyle \min_{ \xi \in A_{\Omega}} C_{\Omega}(\xi)
   ,$ where  $A_{\Omega}$ is the set of harmonic centers of $\Omega.$\begin{flushright}
    $\square$
\end{flushright}

   \end{theorem}
If $\Omega$ is a convex (Euclidean) domain of $\mathbb{H}^1 ,$ then
it has only one harmonic center. As an example, we will give the
case of a Koranyi's ball.
   \begin{example}

   In the case of the Koranyi's ball  $ B(\xi, R) $, of center  $ \xi \in \mathbb{H}^1 $ and radius $R$, using the result of
    Example $ \ref{examplealpha}$ and  Theorem $\ref{eigenvalue},$ we obtain

   $$ \lambda_{1}(B(\xi, R)) \leq C_{\Omega}(\xi)\displaystyle \frac {l_{11}^2}{ R^2}=  16 (\displaystyle\frac{3}{\pi})^{\frac{1}{2}}\displaystyle \frac {l_{11}^2}{ R^2}\\\\$$
\begin{flushright}
    $\square$
\end{flushright}
\end{example}

\par As a continuation of this work, we can propose these open
questions:
   \begin{enumerate}
   \item Compare the first  eigenvalue of $(P_{\Omega}),$   with the first
   eigenvalue of $(P_{\Omega'}),$ where $\Omega'$ is an isoperimetric
   region having  the same volume as $\Omega.$
    \item  Compare  the first
   eigenvalue of $(P_{\Omega'}),$ where $\Omega'$ is an isoperimetric
   region with the first eigenvalue
   of $(P_{\mathbb{B} (0, R)}),$  $\mathbb{B}(0, R)$  is the bubble set with the same volume as
   $\Omega'.$\\
   \item After completing questions 1. and 2., one can attempt  to find an analogous
   formula for the one given by Hersch in the Euclidian settings by using perhaps different
   methods.

\end{enumerate}

\textbf{E-mails:}  ngamara@taibahu.edu.sa, najoua.gamara@fst.rnu.tn\\
        \;\;\;\;\; akremmakni@gmail.com

\textbf{Addresses:} College of Sciences, Taibah University, Saudi
Arabia.\;\;\;\;\;\;\; \\ University Tunis El Manar, F.S.T,
Mathematics Department,
          Tunisia.

\end{document}